\documentclass[12pt]{amsart}

\usepackage{hyperref}

\newtheorem{theorem}{Theorem}[section]

\newtheorem{lemma}[theorem]{Lemma}

\theoremstyle{definition}

\theoremstyle{remark}
\newtheorem*{remark}{Remark}

\numberwithin{equation}{section}

\usepackage{amsmath}
\usepackage{amsthm}
\usepackage{amsfonts}
\usepackage{amssymb}
\usepackage{amscd}
\usepackage{stmaryrd}
\usepackage{hyperref}
\usepackage{verbatim}

\newcommand{\R}{\ensuremath \mathbb{R}}
\newcommand{\C}{\ensuremath \mathbb{C}}
\newcommand{\Z}{\ensuremath \mathbb{Z}}

\newcommand{\inv}{\ensuremath ^{-1}}

\newcommand{\Tr}{\ensuremath \text{Tr}}

\begin{document}

\title{Integral Traces of Weak Maass Forms of Genus Zero Odd Prime Level}
\author{Nathan Green}
\author{Paul Jenkins}

\begin{abstract}
Duke and the second author defined a family of linear maps from spaces of weakly holomorphic modular forms of negative integral weight and level 1 into spaces of weakly holomorphic modular forms of half integral weight and level 4 and showed that these lifts preserve the integrality of Fourier coefficients.  We show that the generalization of these lifts to modular forms of genus 0 odd prime level also preserves the integrality of Fourier coefficients.
\end{abstract}

\maketitle

\section{Introduction and Statement of Results}\label{sec1}
In 1998, Zagier \cite{Zagier} defined a family of lifts for weakly holomorphic modular forms.  He first showed that the traces of the Klein $j$-function, a weight 0 modular form, over CM points in the upper half plane give the coefficients of a modular form of weight $3/2$.  He then extended this to a full basis for the space of weakly holomorphic modular forms of weight 0, showing that the traces of any such weight 0 form are coefficients of a weight $3/2$ form.  This map is called the Zagier lift.  In the same paper, Zagier also gave several generalizations of this lift, notably to spaces of modular forms of weights $-2,-4,-6,-8$ and $-12$.\\

In 2008, Duke and the second author \cite{Jenkins} used Poincar\'e series to generalize these Zagier lifts to weakly holomorphic modular forms of level 1 in all negative weights $2-2s$ for $s\geq 2$.
The image of the lift is a weakly holomorphic modular form of weight $1/2+s$ or $3/2-s$, depending on the parity of $s$.  They expressed this image explicitly in terms of canonical basis elements.  A key consequence of their techniques is that these generalized Zagier lifts preserve the integrality of the Fourier coefficients of modular forms; if the original form has integral Fourier coefficients, then so will its image under the lift.\\

Using a different technique, Miller and Pixton \cite{Miller} also generalized the Zagier lift to negative weights, allowing them to extend the lift to higher levels as well.  Recent work by Alfes \cite{Alfes} has shown that the Zagier lift of a form with integral coefficients in any level will have rational coefficients with bounded denominators.  In this paper, we give a sharper bound on these denominators in the case of genus 0 odd prime levels, showing that the Zagier lifts actually preserve the integrality of coefficients.\\

Let $M_k(N)$ be the space of holomorphic modular forms of weight $k\in \frac{1}{2}\Z$ and level $\Gamma_0(N)$.  Note that if $k=s+1/2$ is half integral, then $N$ must be a multiple of 4.  Further, let $M^!_k(N)$ be the space of weakly holomorphic modular forms of weight $k$ and level $\Gamma_0(N)$.  Recall that a weakly holomorphic modular form is required to be holomorphic on the upper half plane, but is allowed to have poles at the cusps.  For $k\in \Z$, we define a subspace $M^\sharp_k(N)\subset M^!_k(N)$ which is the subspace of weakly holomorphic modular forms which are allowed to have poles only at the cusp at $(\infty)$.  Note that for forms on $\Gamma_0(1)$ these two spaces are equal; $M^!_k(1)=M^\sharp_k(1)$.  For half integral weights $k=s+1/2$ we define related subspaces $M^{\sharp+}_k(N)\subset M^{!+}_k(N)\subset M^!_k(N)$ of forms, where forms in $M^{!+}_k(N)$ are weakly holomorphic satisfying the plus space condition and forms in $M^{\sharp+}_k(N)$ are weakly holomorphic satisfying the plus space condition and have poles only at the cusp $(\infty)$ and its related cusps (see section 2).  A form $f$ satisfies the plus space condition if it has a Fourier expansion supported only on exponents congruent to $0,(-1)^s\pmod 4$, namely $$f(z)=\sum_{n\equiv 0,(-1)^s\pmod 4}a_nq^n,$$ where $q=e^{2\pi i z}$.\\

We let $d$ be an integer with $d\equiv 0,1\pmod 4$ and $D$ be an integer which is a fundamental discriminant (possibly 1).  Suppose that $dD<0$ and that $f$ is $\Gamma_0(N)$ invariant on the upper half plane.  We define the twisted trace of $f$ as
$$\Tr_{d,D}(f)=\sum_Q w_Q\inv \chi(Q)f(\tau_Q),$$
where the sum is over a complete set of $\Gamma_0(N)$ inequivalent positive definite integral quadratic forms $Q(x,y)=ax^2+bxy+cy^2$ with discriminant $dD=b^2-4ac$ and $a\equiv 0\pmod N$, and where
$$\tau_Q=\frac{-b+\sqrt{dD}}{2a}$$
is the associated CM point.  Here $w_Q$ is the cardinality of the stabilizer of $Q$ in $PSL_2(\Z)$ and
$$\chi(Q)=\chi(a,b,c)=\begin{cases}
\chi_D(r),& \text{if }(a,b,c,D)=1\text{ and }Q\text{ represents }r,\text{ where }(r,D)=1\\
0& \text{if }(a,b,c,D)>1,
\end{cases}$$
where $\chi_D$ is the Kronecker symbol $\left (\frac{D}{\cdot}\right )$.\\

We also need a weight raising differential operator for the variable $z=x+iy$.  We define
$$\mathcal{D}=\frac{1}{2\pi i}\frac{d}{dz}=q\frac{d}{dq},$$
and from this define the operator
$$\partial_k=\mathcal{D}-\frac{k}{4\pi y}.$$
We use this operator to define a weight raising operator, which takes weakly holomorphic modular forms of weight $2-2s$ to weak Maass forms of weight 0 (see page 162 of \cite{123}), by
$$\partial^{s-1}=(-1)^{s-1}\partial_{-2}\circ\partial_{-4}\circ\cdots\circ\partial_{4-2s}\circ\partial_{2-2s}.$$
For ease of notation we set $\hat s=s$ if $(-1)^sD>0$ and $\hat s=1-s$ otherwise.  We also set
$$\Tr^*_{d,D}(f)=(-1)^{\lfloor\frac{\hat s-1}{2}\rfloor}|d|^{-\frac{\hat s}{2}}|D|^{\frac{\hat s-1}{2}}\Tr_{d,D}(\partial^{s-1}f).$$

Using this, if we write the Fourier expansion of the modular form $f\in M^{!}_{2-2s}(N)$ as $f(z)=\sum a(n)q^n $, then we generalize the definition of the Zagier lift of a form $f(z)$ from \cite{Jenkins} by setting
\begin{equation}\label{zaglift}
\mathfrak{Z}_D f(z)=\sum_{m>0}a(-m)m^{s-\hat{s}}\sum_{n|m}\chi_D(n)n^{\hat s-1}q^{-\frac{m^2|D|}{n^2}}+C+\sum_{d:dD<0}\Tr^*_{d,D}(f)q^{|d|},
\end{equation}
where $C$ is the unique constant term which causes $\mathfrak{Z}_D(f)$ to be modular as described in (Theorem~\ref{theorem:ours}).  We note that $\mathfrak{Z}_D$ has an implicit dependence on the weight and level of the form it acts on.\\

We now state the main theorem from \cite{Jenkins}.

\begin{theorem}[\cite{Jenkins},Theorem 1]\label{theorem:jenkins}
Suppose that $f\in M^!_{2-2s}(1)$ for some integer $s\geq 2$.  If $D$ is a fundamental discriminant with $(-1)^sD>0$, then $\mathfrak{Z}_D f\in M^{!+}_{3/2-s}(4)$, while if $(-1)^sD<0$, then $\mathfrak{Z}_D f\in M^{!+}_{s+1/2}(4)$.  If $f$ has integral Fourier coefficients, then so does $\mathfrak{Z}_D f$.
\end{theorem}

The main result of this paper establishes the integrality preserving properties of the lift for odd prime levels with genus 0.  We will denote these primes by $p$ throughout the remainder of the paper.  The genus 0 odd primes are $\{3,5,7,13\}$.

\begin{theorem}\label{theorem:ours}
Suppose that $f\in M^\sharp_{2-2s}(p)$ for some integer $s\geq 2$ and $p$ a genus 0 odd prime.  If $D$ is a fundamental discriminant with $(-1)^sD>0$, then $\mathfrak{Z}_D f\in M^{!+}_{3/2-s}(4p)$, while if $(-1)^sD<0$, then $\mathfrak{Z}_D f\in M^{!+}_{s+1/2}(4p)$.  If $f$ has integral Fourier coefficients, then so does $\mathfrak{Z}_D f$, with the exception that if $s+1/2 = 5/2+6n$ for $n\in \Z$ and $p=7$ then $\mathfrak{Z}_D f$ has at most a single power of $2$ in the denominators of its Fourier coefficients.
\end{theorem}

\noindent We note that the initial part of theorem (\ref{theorem:ours}) can be deduced from \cite{Miller}.  In this paper, we give an independent proof which allows us to prove the statement on integrality.  We discuss the exception to the theorem in greater detail in section \ref{sec3}.\\

We give some examples of Theorem \ref{theorem:ours} before proceeding to the proof of the theorem.  Let $$E_2^{(5)}(z)=\tfrac{1}{4}(5E_2(5z)-E_2(z))$$ be the normalized weight 2 Eisenstein series of level 5, where $E_2(z)$ is the usual non-modular Eisenstein series of weight 2.  Further, let $$H_4^{(5)}(z)=\frac{\eta(5z)^{10}}{\eta(z)^2},$$ be the unique form in $M_4(5)$ with leading term $q^2$, where $\eta(z)$ is the usual Dedekind $\eta$-function.  Define the weakly holomorphic modular form $f(z) \in M_{-2}^{\sharp}(5)$ by $$f(z)=\frac{E_2^{(5)}(z)}{H_4^{(5)}(z)}=q^{-2} + 4q^{-1} + 5 - 16q - 11q^{2} - 48q^{3} + 134q^{4} + 80q^{5} - O(q^6).$$  If we let $D=1$, then the first few nonzero traces of $f$ are $$\Tr^*_{-4,1}(f)=\frac{-1}{4}\left  (\frac{1}{2}\partial f(\tfrac{-2+i}{5})+\frac{1}{2}\partial f(\tfrac{-3+i}{5} )  \right )=-1,$$
$$\Tr^*_{-11,1}(f)=\frac{-1}{11}\left  (\partial f(\tfrac{-3+i\sqrt{11}}{10})+\partial f(\tfrac{-7+i\sqrt{11}}{10} )  \right )=-16,\text{ and}$$
$$\Tr^*_{-15,1}(f)=\frac{-1}{15}\left  (\partial f(\tfrac{-5+i\sqrt{15}}{10})+\partial f(\tfrac{-5+i\sqrt{15}}{20}  ) \right )=10.$$  Thus the image of the Zagier lift of $f$ is $$\mathfrak{Z}_1 f = q^{-4} + 6q^{-1} + 5 - q^{4} - 16q^{11} + 10q^{15} - O(q^{16})\in M^!_{-1/2}(20).$$

Similarly, let $$E_2^{(7)}(z)=\tfrac{1}{6}(7E_2(7z)-E_2(z))$$ be the normalized weight 2 Eisenstein series of level 7 and let $$H_6^{(7)}=\eta(7z)^{10}\eta(z)^2\Phi_7(z)$$ be the form in $M_6(7)$ with leading term $q^4$.  Here, $\Phi_7(z)$ is a level 7 Hauptmodul; see section \ref{sec3} for a definition.  Define the weakly holomorphic modular form $g(z) \in M_{-4}^{\sharp}(7)$ by $$g(z)=\frac{E_2^{(7)}(z)}{H_6^{(7)}(z)}=q^{-4} + 2q^{-3} + 3q^{-2} + -10q^{-1} - 7
- 18q + 22q^{2} - 14q^{3} + 114q^{4} + 104q^{5} - O(q^6).$$  If we let $D=-3$, then the image of the Zagier lift of $g$ is
$$\mathfrak{Z}_{-3} g = q^{-48} + 2q^{-27} -q^{-12} -6q^{-3}- 266  + 426q - 2661q^{4} + 19542q^{8} + O^{9}\in M_{-3/2}^!(28).$$

The rest of the paper proceeds as follows.  In section \ref{sec2} we discuss the existence of poles of certain Poincar\'e series which are used in the main proof of our theorem and we study poles of half integral weight modular forms in the plus space in general.  In section \ref{sec3} we describe canonical bases for the spaces $M^{\sharp+}_{s+1/2}(4p)$.  In section \ref{sec4} we prove a duality principle for the coefficients of these basis elements, and in section \ref{sec5} we give the proof of Theorem \ref{theorem:ours}.

\section{Poles of Poincar\'e series}\label{sec2}
We require certain Poincar\'e series for the proof of the main theorem.  In this section we discuss the poles of these Poincar\'e series as well as poles in general for the plus space $M_{s+1/2}^{! +}(p)$.\\

We first give some notation related to the projection operator which is defined by Kohnen in \cite{Kohnen}.  We will use this projection operator to define the Poincar\'e series.
Let $G$ be the metaplectic group (also defined in \cite{Kohnen}) and let $A=\left(M,\phi(z)\right)\in G$ act on a form $g(z)$ by
\begin{equation}\label{slashop}
g(z)|A=g(Mz)\phi(z)\inv.
\end{equation}  The projection operator acting on a form $g$ is defined by

\begin{equation}
\label{projoperator} g(z)|pr=(-1)^{\lfloor (s+1)/2\rfloor}\frac{1}{3\sqrt{2}}\left(\sum_{v=-1,0,1,2}g|\xi A^*_v\right)+\frac{1}{3} g,
\end{equation}
where $\xi=\left(\left(\begin{smallmatrix} 4&1\\0&4\end{smallmatrix}\right),e^{(2s+1)\pi i/4}\right)\in G$, $A_v=\left(\begin{smallmatrix} 1&0\\4pv&1\end{smallmatrix}\right)\in \Gamma_0(4p)$, and $A_v^*$ is the projection of $A_v$ into the metaplectic group $G$.  The sum is over conjugacy classes mod 4, so it suffices to sum over the numbers $-1,0,1,2$.\\

The Poincar\'e series we use in the proof of the main theorem are described in detail by Bringmann and Ono in \cite{OnoArithmetic} and are generalized to higher levels by Miller and Pixton in \cite{Miller}.  We describe them briefly here.  For a real variable $y\in \R-\{0\}$, for weight $k=s+1/2$ and $w\in \C$, define
$$\mathcal{M}_w(y)=|y|^{-k/2}M_{\tfrac{k}{2}\,\text{sgn}(y),w-1/2}(|y|),$$
where $M_{v,\mu}(z)$ is the usual $M$-Whittaker function.  For $m\geq 1$ with $(-1)^{s+1}m\equiv 0,1\pmod 4$, define
$$\phi_{-m,w}(z) = \mathcal{M}_w(-4\pi m \text{Im}(z))e(-m\text{Re}(z)).$$
Then, for $4|N$, define the Poincar\'e series
$$\mathcal{F}_{s+1/2,-m}(N,w;z) = \sum_{A\in \Gamma_\infty \backslash \Gamma_0(N)} (\phi_{-m,w}|_k A)(z),$$
where $\Gamma_\infty$ is the set of translations in $SL_2(\Z)$.  Finally, we define
$$F_{s+1/2,-m}^{(N)}(z)=\begin{cases} \frac{3}{2}\mathcal{F}_{s+1/2,-m}(N,\tfrac{k}{2};z)|pr & \text{if } s\geq 1\\ \tfrac{3}{2(1-k)\Gamma(1-k)} \mathcal{F}_{s+1/2,-m}(N,1-\tfrac{k}{2};z)|pr & \text{if }s\leq 0.\end{cases}$$
If $s>0$, then $\phi_{-m,k/2}(z) = e^{-2\pi i m z}$ and $F_{s+1/2,-m}^{(N)}(z)\in M^{\sharp +}_{s+1/2}(N)$.  Otherwise, $F_{s+1/2,-m}^{(N)}(z)$ is a weak Maass form.  We refer the reader to \cite{OnoArithmetic} for more information on this case; we will use only the first case in this paper.\\

In order to examine the holomorphicity of $F_{s+1/2,-m}^{(N)}(z)$ at the cusps, we need to understand how acting on a modular form with the projection operator affects its values at the cusps in general.  We will denote the Fourier expansion of $g$ at the cusp $(\frac{a}{b})$ by $g^{(\frac{a}{b})}(z)$.  Calculating the value of the form $g(z)|pr$ at the cusp $(\frac{a}{b})$ involves the calculation
$$\lim_{z\rightarrow i\infty} (g(z)|pr)|\eta^{(\frac{a}{b})},$$
where $\eta^{(\frac{a}{b})}\in G$ is the element of the metaplectic group consisting of the matrix which maps the cusp at $(\infty)$ to the cusp at $(\frac{a}{b}),$ together with the correct automorphy factor.  We note that the cusps of $\Gamma_0(4p)$ are $(\infty),(\frac{1}{p}),(\frac{1}{2p}),(0),(\frac{1}{2}),(\frac{1}{4})$.  We examine the question of when the form $g(z)|pr$ is holomorphic at the cusp $(0)$, and apply this to the Poincar\'e series $F_{s+1/2,-m}^{(N)}(z)$.  For this calculation, we must act on each term of (\ref{projoperator}) by $\eta^{(0)}$.  We will give the calculation for the case $p=3$ and note that the cases for the other genus 0 odd primes are nearly identical.\\

Following Kohnen's technique in \cite{Kohnen} on page 251, we act first on the sum of the terms
$$g|\xi A_0^*+g|\xi A_2^*=g|\xi+g|\xi\inv.$$

When we operate on these two terms with $\eta^{(0)}$ and combine group elements of $G$, we get
\begin{equation}\label{firstpice}
g|(\xi+g|\xi\inv)|\eta^{(0)}=Cz^{-s-1/2}g\left(\frac{-1}{z}+\frac{1}{4}\right)+Cz^{-s-1/2}g\left(\frac{-1}{z}-\frac{1}{4}\right),
\end{equation}
where throughout the calculation $C$ will represent arbitrary constants.
Both of the forms in the right hand side of \ref{firstpice} are equivalent to $g^{(0)}(z)$ twisted by a character.  Since twisting by a character does not change the order of vanishing (or order of the pole) at a cusp, the sum of the first two terms $g|\xi A_0^*+g|\xi A_2^*$ is holomorphic (has a pole) at the cusp $(0)$ if and only if $g$ is holomorphic (has a pole) at $(0)$.\\

Next we operate on the term $g|\xi A_1^*$ from (\ref{projoperator}) in a similar way.  We first separate out a piece of this which is in the projection of $\Gamma_0(12)$ into $G$ and thus acts trivially on $g$.  From there we get
\begin{align*}
g|\xi A_1^*\eta^{(0)}&=g|\left(\left(\begin{smallmatrix}1&-16\\4&-48\end{smallmatrix}\right),C(z-12)^{s+1/2}\right)\\
&=g|\left(\left(\begin{smallmatrix}1&0\\4&1\end{smallmatrix}\right),(4z+1)^{s+1/2}\right)\left(\left(
\begin{smallmatrix}1&-16\\0&16\end{smallmatrix}\right),C\right)\\
&=Cg^{(\frac{1}{4})}\left(\frac{z}{16}-1\right).
\end{align*}
This shows that this term is holomorphic (has a pole) at $(0)$ if and only if $g$ is holomorphic (has a pole) at the cusp $\left(\frac{1}{4}\right)$.\\

For the next term of (\ref{projoperator}), $g|\xi A_{-1}^*$, we proceed almost identically.  The calculation is
\begin{align*}
g|\xi A_{-1}^*\eta^{(0)}&=g|\left(\left(\begin{smallmatrix}1&-8\\8&-48\end{smallmatrix}\right),C(z-6)^{s+1/2}\right)\\
&=g|\left(\left(\begin{smallmatrix}1&0\\8&1\end{smallmatrix}\right),(8z+1)^{s+1/2}\right)\left(\left(\begin{smallmatrix}
1&-8\\0&16\end{smallmatrix}\right),C\right)\\
&=Cg^{(\frac{1}{8})}\left(\frac{z}{8}-\frac{1}{2}\right).
\end{align*}
Since the cusp at $(\frac{1}{8})$ is equivalent to the cusp at $(\frac{1}{4}),$ this tells us that $g|\xi A_{-1}^*$ is holomorphic (has a pole) at $(0)$ if and only if $g$ is holomorphic (has a pole) at $(\frac{1}{4})$.\\

Finally, we act on the term $\frac{1}{3}g$ from (\ref{projoperator}).  This becomes simply $\frac{1}{3} g^{(0)}(z)$ which is exactly the Fourier series at the cusp $(0)$.  Thus we conclude that $g(z)|pr$ is holomorphic at the cusp $(0)$ if and only if the form $g(z)$ is holomorphic at the cusps $(0)$ and $(\frac{1}{4})$.  This concludes the calculation for level 3.  (We will see later that the value of a form in the plus space at the cusps $(0)$ and $(\frac{1}{4})$ is related to the value at the cusp $(\frac{1}{2})$ and that these three cusps are grouped together.  The order of a pole at any one of these cusps determines the order of the poles at the rest of them.)\\

Before we apply this analysis to the Poincar\'e series $F_{s+1/2,-m}^{(N)}(z)$, we will discuss where a general form $f(z)\in M_k^{\sharp +}(4p)$ is allowed to have poles.  Since $f(z)$ is already in the plus space, $$f(z)|pr=f(z).$$  Again, from Kohnen's work in \cite{Kohnen}, $$f(z)|pr=\frac{2}{3}[f(z)+\epsilon_1 f^{(\frac{1}{p})*}(4z)+\epsilon_2 f^{(\frac{1}{2p})*}(4z)]$$ where the $^*$ indicates that we are twisting by a character and $\epsilon_1,\epsilon_2$ are constants.  By examining the behavior of $f(z)$ at each cusp, this calculation shows that if $f(z)$ has a pole of order $l$ at $(\infty)$, then $f^{(\frac{1}{p})*}(4z)$ and $f^{(\frac{1}{2p})*}(4z)$ also have poles of order $l$ at $(\infty)$.  Thus the poles at the cusps $(\infty),(\frac{1}{p}),(\frac{1}{2p})$ are, in a sense, grouped together and determine each other's orders.\\

A calculation similar to the one above shows that operating on the projection operator with $\eta^{(\frac{1}{2})}$ gives $$f|pr|\eta^{(\frac{1}{2})}=\frac{2}{3}\left[f^{(\frac{1}{2})}+\epsilon_1 f^{(\frac{1}{4})*}+\epsilon_2 f^{(0)*}\right]$$ showing that the cusps of $(0),(\frac{1}{2}),(\frac{1}{4})$ are similarly grouped together.  Acting on $f|pr$  with the other operators $\eta^{(\frac{a}{b})}$ gives a similar result.  Therefore, if a form $f\in M_k^{\sharp +}(4p)$ is holomorphic (has a pole) at one of the three cusps $(0),(\frac{1}{2}),(\frac{1}{4})$ then it must be holomorphic (have a pole) at the other two as well.  The same is true for the cusps $(\infty),(\frac{1}{p}),(\frac{1}{2p})$.\\

We now prove the following theorem about the Poincar\'e series described above.  We note that we will use the notation described in the theorem throughout the rest of the paper.

\begin{theorem} \label{theorem:poincare}
Given a genus 0 odd prime $p$, $s\in \Z$ with $s\geq 2$ and any nonzero $m\in \Z$ with $(-1)^sm\equiv 0,1\pmod 4$ there exists a Poincar\'e series $$F_{s+1/2,m}^{(4p)}=q^m+\sum_{n\geq 0,n\equiv 0,(-1)^s\pmod 4}p_{s+1/2}^{(4p)}(m,n)q^n$$ which is a weakly holomorphic modular form of weight $s+1/2$ of level $4p$.  Each of these forms is holomorphic (and actually vanishes) at the cusps $(0),(\frac{1}{2}),(\frac{1}{4})$ and has a pole at the cusp $(\infty)$ with related poles at the cusps $(\frac{1}{p}),(\frac{1}{2p})$.
\end{theorem}

\begin{proof}
A generalization of theorem 3.3 of \cite{BruinierHilbert} shows that the Poincar\'e series $\mathcal{F}_{s+1/2,m}$ is holomorphic (and actually vanishes) at all cusps except $(\infty)$.  Thus, following the above discussion, we may conclude that the form $F_{s+1/2,m}^{(4p)}=\mathcal{F}_{s+1/2,m}|pr$ is holomorphic (and actually vanishes) at the cusp $(0)$ and the related cusps $(\frac{1}{2})$ and $(\frac{1}{4})$ and has a pole at the cusp $(\infty)$ and its related cusps $(\frac{1}{p})$ and $(\frac{1}{2p})$.  We note that Ono mentions this generalization as well as some additional information about Poincar\'e series in section 5.3 of \cite{OnoAWS}.\\
\end{proof}

\section{Canonical Bases for level 4p Forms in the Plus Space}\label{sec3}
Canonical bases for integer weight spaces $M^!_k(p)$ are well understood.  In the proof of our main theorem, we will use a simple generalization of such bases as described in \cite{Garthwaite}, extended to all genus 0 primes.  We will also require canonical bases for the spaces $M^{\sharp+}_{s+1/2}(4p)$ for $s\in \Z$ which have special properties.  We describe these bases here.
\begin{theorem} \label{theorem:bases}
There exist explicit bases of modular forms $f_{s+1/2,m}^{(4p)}$ for the spaces $M^{\sharp+}_{s+1/2}(4p)$ for $s\in \Z$ which have a Fourier expansion of the form $$f_{s+1/2,m}^{(4p)}=q^{-m}+\sum_{n=n_0}^\infty a(n)q^n$$ where $-m,n\equiv 0,(-1)^{s}\pmod 4$ and $n_0$ is the value that causes these basis elements to have the largest possible gap between the exponent of the leading term and the exponent of the next term, and which have integral Fourier coefficients with the one exception that if $s+1/2 = 5/2+6r$ for $r\in \Z$ and $p=7$, then the Fourier coefficients may have a single power of 2 in the denominator.
\end{theorem}

\begin{proof}
We first define some auxiliary forms.  Let the modular theta function be $$\theta(z) =\sum_{n\in \Z} q^{n^2}.$$  Let $C_{6,3}(z)$ be the unique normalized cusp form of weight $6$ and level $\Gamma_0(3)$, let $C_{4,5}(z)$ be the unique normalized cusp form of weight 4 and level $\Gamma_0(5)$, let $C_{6,7}(z)$ be the normalized cusp form of weight 6 and level $\Gamma_0(7)$ which has leading term $q^3$, and let $C_{6,13}(z)$ be the normalized cusp form of weight 6 and level $\Gamma_0(13)$ with leading term $q^3$.  Also let $\eta(z)$ be the Dedekind eta function and let $$\Phi_p(z)=\left(\frac{\eta(pz)}{\eta(z)}\right)^{\tfrac{24}{(p-1)}}\in M_0^!(p)$$ be a level $p$ Hauptmodul.  Note that $\Phi_p(z)$ has a simple pole at the cusp $(0)$ and a simple zero at the cusp $(\infty)$.  Further, let $$\Psi_p(z)=\frac{1}{\Phi_p(z)}\in M_0^!(p),$$ which swaps the pole and zero.\\

To construct these bases for $M_{s+1/2}^{\sharp +}(4p)$ we start by describing a basis for the holomorphic space $M_{s+1/2}^{+}(4p)$, then use this basis to construct a basis for the full weakly holomorphic space $M_{s+1/2}^{\sharp +}(4p)$.  We begin by calculating a spanning set of Fourier expansions for holomorphic modular forms for the spaces $M_{s+1/2}(4p)$ using the techniques of modular symbols.  We calculate these explicitly for forms of each half integral weight space from $M_{11/2}(12)$ up to and including $M_{21/2}(12)$, from $M_{11/2}(20)$ up to and including $M_{17/2}(20)$, from $M_{11/2}(28)$ up to and including $M_{21/2}(28)$, and from $M_{7/2}(52)$ up to and including $M_{13/2}(52)$.  We row reduce this spanning set to get a set of forms each of which begins with $q^m+\sum_{n\geq n_0} a(n) q^n$ for each $m$ in the set $\{0,1,...,\dim(M_{s+1/2}(4p))-1\}$, where $n_0$ gives the largest possible gap between the first two exponents of the Fourier expansion.  The Magma software package \cite{Magma} has an especially good implementation for half integral weight modular forms which we used extensively in our calculations.  We also used SAGE \cite{sage} extensively.  \\

We calculate the dimension of each Kohnen plus space $M^+_{s+1/2}(4p)$ for the weights listed above using isomorphisms given by Ueda in \cite{Ueda}.
We then take linear combinations of forms in the row reduced spanning set for each weight and level to create a set of forms which span the plus space.  Call these forms $f_{k,m}^{(4p)}$.  We calculate that, up to a given exponent, the Fourier coefficients of each $f_{k,m}^{(4p)}$ are integral with the exception of $M_{17/2}(28)$, which has single powers of 2 in the denominators of the Fourier coefficients of some of the basis elements.  This exception shows up in infinitely many weights, specifically whenever $s+1/2 = 5/2+6n$ for $n\in \Z$.

\begin{remark}
For example, in $M_{-7/2}^{\sharp +}(28)$, the basis element with leading exponent $q^{-12}$ has a Fourier expansion
$$q^{-12} + \frac{1}{2}q^{-11} -\frac{3}{2}q^{-8} + 2q^{-7} + \frac{1}{2}q^{-4}
-4q^{-3} + \frac{5}{2} - 2q - 7q^{4} - \frac{5}{2}q^{5} -
2q^{8}+O(q^{9}).$$
We note that in most cases this exception will not actually affect the integrality of the Zagier lift due to theorem 6.1 of Alfes \cite{Alfes}, which describes certain conditions under which the image of the lift must be $2$-integral.  Computationally, however, we have found several instances where basis elements containing the single powers of 2 in the denominator do not meet the conditions of Alfes' theorem, and thus they may actually occur in the image of the lift.  This is a topic we wish to study further.\\
\end{remark}

We will prove that this finite calculation of the Fourier expansions of these basis elements suffices to show that all the Fourier coefficients of these basis elements are integral (or half-integral in the case of the exception).  For weights with $s$ odd, we multiply each $f_{k,m}^{(4p)}$ by the modular theta function $\theta(z)$.  The function $\theta(z)$ is a weight $1/2$ holomorphic modular form which is non-vanishing on the upper half plane and holomorphic at the cusps of $\Gamma_0(4)$.  The resulting product will be a holomorphic modular form of even integer weight of level $4p$.  We then examine the resulting Fourier expansions and calculate using the Sturm bound that they are known holomorphic modular forms of even integral weight which have integer coefficients and a leading coefficient of 1.  This shows that the original $f_{k,m}^{(4p)}$ also has integral Fourier coefficients.  If $s$ is even, then we repeat the same process, but multiply by $\theta(z)^3$ instead.\\

To construct a basis of forms of weight $k+6$ in level $12$ we multiply the form $f_{k,m_0}^{(12)}$ of weight $k$ with the highest exponent in the leading term by the form $C_{6,3}(4z)\Phi_3(4z)$, which is holomorphic at the cusp $(0)$ and its related cusps and vanishes with order 8 at the cusp $(\infty)$ and its related cusps.  We note that multiplying $C_{6,3}(4z)$ by $\Phi_3(4z)$ maximizes the order of vanishing at the cusp $(\infty)$ and cancels the zeros at the cusp $(0)$.  We do the same for the form $f_{k,m_1}^{(12)}$ with the second highest exponent in the leading term.  One of these exponents will be $0\pmod 4$ and the other will be $(-1)^s\pmod 4$.  We then multiply each of these forms by successive powers of $\Psi_3(4z)$, which has a leading term of $q^{-4}$ in its Fourier expansion, to decrease the leading exponent until we obtain a spanning set of forms.  We then row reduce to get basis elements with the largest possible gap in the Fourier expansion.  In level $20$ we do the same, but multiply by $C_{4,5}(4z)\Phi_5(4z)$, then by successive powers of $\Psi_5(4z)$.  In level $28$ we multiply by $C_{6,7}(4z)\Phi_7(4z)$ and then by powers of $\Psi_7(4z)$, and in level $52$ we multiply by $C_{4,13}(4z)\Phi_{13}(4z)$ and then by powers of $\Psi_{13}(4z)$.\\

Using these holomorphic forms $f_{k,m}^{(4p)}$ as a starting point, we form the spaces $M_{s+1/2}^{\sharp +}(4p)$.  Note that each $f_{k,m}^{(4p)}$ has non-zero coefficients only for exponents which are congruent to $0,(-1)^{s}\pmod{4}$.  Multiplying one of these basis elements by $\Psi_p(4z)$ will create a form with leading term 4 powers lower which is still in the plus space.  Thus we can get a spanning set of forms with arbitrarily large negative exponents in the leading term by multiplying by successive powers of $\Psi_p(4z)$ and row reducing.\\

We get bases of negative weight $s+1/2-6$ forms in levels $12,20,28$ and $52$ by dividing the bases of weight $s+1/2$ by the forms $C_6(4z)\Phi_3(4z)$, $C_{4,5}(4z)\Phi_5(4z)$, $C_{6,7}(4z)\Phi_7(4z)$, and $C_{4,13}(4z)\Phi_{13}(4z)$ respectively and then row reducing.  This allows us to create bases of modular forms with integer Fourier coefficients for any half integral weight with arbitrarily large negative leading exponent.
\end{proof}

We also note that we can construct a subspace of $M_{s+1/2}^{\sharp +}$ where each element vanishes at the cusp $(0)$ and at its related cusps $(\frac{1}{2})$ and $(\frac{1}{4})$.  To do this, we use the same procedure as above, except we multiply each form by an extra copy of $\Psi_p(4z)$, which vanishes at the cusp $(0)$, before row reducing.  We use these bases in the proof of our main theorem; they are analogues of integer weight bases as described in \cite{Garthwaite}.

\section{Duality of Basis Elements}\label{sec4}
In the proof of the main theorem we make repeated use of duality between coefficients of the basis elements $f_{k,m}^{(4p)}$.
Duality in the integral weight case for odd prime levels of genus 0 is already well understood.  Garthwaite and the second author describe this in \cite{Garthwaite}.  El-Guindy derives similar coefficient duality results in \cite{El} by writing down generating functions for basis elements.\\

We describe some notation for the integral weight case.  Let the set $\{f_{2s,m}^{(p)}\}_{m=m_0}^\infty$ for $s\in \Z$ be a row reduced basis as in section \ref{sec3} for $M_{2s}^{\sharp}(p)$ and let $\{g_{2-2s,l}^{(p)}\}_{l=l_0}^\infty$ be a row reduced basis for the subspace of $M_{2-2s}^{\sharp}(p)$ which vanishes at the cusp $(0)$.  Write
$$f_{2s,m}^{(p)}(z)=q^{-m}+\sum_{n=n_0}^\infty a_{2s}^{(p)}(m,n)q^n$$ and
$$g_{2-2s,l}^{(p)}(z)=q^{-l}+\sum_{n=n_0}^\infty b_{2-2s}^{(p)}(m,n)q^n.$$ From \cite{Garthwaite} we have the following theorem.

\begin{theorem}\label{theorem:dualtiyintegral}
For the bases given above,
$$a_{2s}^{(p)}(m,l)=-b_{2-2s}^{(p)}(l,m).$$
\end{theorem}

A similar duality holds for half integral weights, which we will prove in this section.  Let the set $\{f_{s+1/2,m}^{(4p)}\}_{m=m_0}^\infty$ for $s\in \Z$ be a row reduced basis as in section \ref{sec3} for $M^{\sharp +}_{s+1/2}(4p)$ and let $\{g_{3/2-s,l}^{(4p)}\}_{l=l_0}^\infty$ be a row reduced basis for the subspace of $M^{\sharp +}_{3/2-s}(4p)$ which vanishes at the cusp $(0)$ and its related cusps.  As above, we write
$$f_{s+1/2,m}^{(4p)}(z)=q^{-m}+\sum_{n=n_0}^\infty a_{s+1/2}^{(4p)}(m,n)q^n$$ and
$$g_{3/2-s,l}^{(4p)}(z)=q^{-l}+\sum_{n=n_0}^\infty b_{3/2-s}^{(4p)}(m,n)q^n.$$ We note that we use this notation to describe the coefficients of these basis elements throughout the remainder of the paper.  We have the following theorem.

\begin{theorem}\label{theorem:dualityhalfintegral}
For the bases given above,
$$a_{s+1/2}^{(4p)}(m,l)=-b_{3/2-s}^{(4p)}(l,m).$$
\end{theorem}

\begin{proof}
Given forms $f\in M^{\sharp +}_{s+1/2}(4p)$ and $g\in M^{\sharp +}_{3/2-s}(4p)$ where $g$ vanishes at the cusp $(0)$, the form $fg$ is in the space $ M_2^\sharp(4p)$.  We will show that the constant term of $fg$ vanishes, then apply this to the basis elements $f_{s+1/2,m}^{(4p)}(z)$ and $g_{3/2-s,l}^{(4p)}(z)$ to prove the duality.  The Fourier coefficients of $fg$ vanish on exponents congruent to $2\pmod 4$.  Let the $U_m$ operator be the standard operator which takes a form $f=\sum^{\infty}_{n=n_{0}} a_n q^n$ to $f|U_m=\sum_{n=n_0}^{\infty} a_{nm}q^n$.  Then $fg|U_2\in M_2^!(2p)$, by \cite{Atkin}.  Then, since $fg|U_2$ is supported only on even exponents, $fg|U_4$ is actually in $M_2^!(p)$.  The holomorphic space $M_2(p)$ is one dimensional and its only form does not vanish at the cusp $(0)$.  By the definition of $g$, the form $fg$ vanishes at the cusp $(0)$; we will show that $fg|U_4$ also vanishes at the cusp $(0)$.  The operator $U_4$ may be expressed as the sum of matrices acting on $fg$,
$$fg|U_4=\frac{1}{4}\sum_{j=0}^3 fg\left(\frac{z+j}{4}\right).$$  Acting on this with $|\eta^{(0)}$ gives $$fg|U_4|\eta^{(0)}=\frac{1}{4}z^{-2}\sum_{j=0}^3 f\left(\left(\begin{smallmatrix}j&b\\4&d\end{smallmatrix}\right)\left(\begin{smallmatrix}1&-j\\0&4\end{smallmatrix}\right)\right),$$ where $b,d$ are simply the correct integers which cause the first matrix to be in $SL_2(\Z)$.  Some algebraic manipulation shows that taking the limit of this sum as $z$ goes to $i\infty$ is equal to (up to a few constants) the sum of $fg$ evaluated at several different cusps, each of which is equivalent to $(0)$.  Since the form $fg$ vanishes at the cusp $(0)$, the above sum also vanishes.  Hence $fg|U_4$ vanishes at the cusp $(0)$.\\

Finally, note that derivatives of forms in $M_0^{!}(p)$ are in $M_2^!(p)$.  Thus we can eliminate the principal part of $fg|U_4$ by subtracting a derivative of a suitable polynomial in $\Psi_p(z)$, a Hauptmodul of level $p$, to get a form which is holomorphic and vanishes at the cusp $(0)$ and thus is identically 0.  Since this derivative has no constant term, we conclude that $fg|U_4$ also has no constant term and thus $fg$ has no constant term.\\

Applying this to the basis elements $f_{s+1/2,m_i}^{(4p)}$ and $g_{3/2-s,l_j}^{(4p)}$, the resulting product $f_{s+1/2,m_i}^{(4p)}\cdot g_{3/2-s,l_j}^{(4p)}$ is in $M^\sharp_2(\Gamma_0(4p))$ and vanishes at the cusp $(0)$, and thus its constant term vanishes.  This constant term is $$\sum_{n=-\infty}^\infty a_{s+1/2}^{(4p)}(m,n)b_{3/2-s}^{(4p)}(n,m)=0.$$  Since each basis element has only finitely many negative exponents, this sum is actually finite.  If we examine the explicitly constructed bases as described in section 2, together with dimension formulas for each of the spaces, we see that for each weight $k$, the form with the highest power in the leading term $$f_{s+1/2,m_1}^{(4p)}=q^{-m_1}+\sum_{n=M}^\infty a_nq^n$$ and the form with the highest power in the leading term $$g_{3/2-s,l_1}^{(4p)}=q^{-l_1}+\sum_{n=L}^\infty b_nq^n$$ have $m_1=L$ and $l_1=M$.  This means that if we multiply any two forms together from these bases, there will always be exactly two terms which add up to give the resulting constant coefficient, namely $$a_{s+1/2}^{(4p)}(m,l)+b_{3/2-s}^{(4p)}(l,m)=0.$$  This gives the desired duality.
\end{proof}

\section{Proof of Theorem 1.2}\label{sec5}
We now express the Zagier lift of an arbitrary modular form $f\in M^\sharp_{2-2s}(p)$ in terms of the basis elements $f_{k,m}^{(N)}$ for $k=s+1/2$ or $k=3/2-s$, depending on the parity of $s$.  Since these basis elements have integral Fourier coefficients, this will allow us to prove Theorem \ref{theorem:ours}.  For a genus 0 odd prime $p$ and integer $n$, we will denote $N(n)=\tfrac{4p}{\gcd(n,p)}$ to ease notation, where we note the implicit dependence on $p$, which will always be obvious.  We also use the notation $\chi_1^{(p)}(n)$ to denote the trivial level $p$ character which vanishes if $p|n$ and equals $1$ otherwise.
\begin{theorem}\label{maintheorem}
Let $f\in M^\sharp_{2-2s}(p)$ and let $s\geq 2$ be an integer.  Suppose that $D$ is a fundamental discriminant with $(-1)^sD>0$.  Then the $D$th Zagier lift of $f(z)=\sum a(n)q^n$ is given by
\begin{equation}\label{Fequation}
\mathfrak{Z}_Df=\sum_{m>0} a(-m)\sum_{n|m} \chi_D(n) n^{s-1}f_{3/2-s,\frac{m^2|D|}{n^2}}^{\left (N(n)\right )}.
\end{equation}
\end{theorem}

\noindent We will denote the right hand side of (\ref{Fequation}) by $F$.\\

\begin{proof}

For this proof, we use the integer weight bases $\{f_{2s,m}^{(p)}\}_{m=m_0}^\infty$ and $\{g_{2-2s,m}^{(p)}\}_{m=l_0}^\infty$.  Denote $l=l_0$.  We also use the half integer weight bases $\{f_{s+1/2,m}^{(4p)}\}_{m=j_0}^\infty$ and $\{g_{3/2-s,m}^{(4p)}\}_{m=k_0}^\infty$ described in sections \ref{sec3} and \ref{sec4}, the half integral weight bases $\{f_{s+1/2,m}^{(4)}\}_{m=n_0}^\infty$ described in \cite{Jenkins}, as well as the Poincar\'e series $F_{k,m}^{(N)}$ described in section 2.  We will also repeatedly use the modified Shimura lift as defined by Kohnen in \cite{Kohnen}, which we give here.  Let the level $4p$ modified Shimura lift twisted by the fundamental discriminant $D$ act on a form $f(z)=\sum_{n=1}^\infty a(n)q^n \in S^{+}_{s+1/2}(4N)$ by
\begin{equation}\label{shimdef}
\mathfrak{S}_D^{(4p)}(f(z)) = \sum_{n=1}^{\infty} \left (\sum_{d|n} \chi_1^{(p)}(n)\chi_D(n)d^{s-1}a\left (\frac{n^2|D|}{d^2} \right )\right) q^n.
\end{equation}
Kohnen shows that the image of the lift $\mathfrak{S}_D^{(4p)}(f(z))$ is a cusp form in $S_{2s}(p)$.  We will complete the proof by comparing the positive powers and the principal parts of $\mathfrak{Z}_D f$ and $F$ separately.  We begin by comparing the positive powers.\\

By combining theorems (1.1) and (1.2) of Miller and Pixton \cite{Miller}, for $m\geq l$, we can write the traces as sums of coefficients of the Poincar\'e series from section \ref{sec2} as
\begin{align*}
-\Tr^*_{d,D}(f_{2-2s,m}^{(p)})=&\sum_{n|m}\chi_D(n)n^{s-1}p_{s+1/2}^{(N(n))}\left (-|d|,\tfrac{m^2|D|}{n^2}\right )\\
&+\sum_{j=1}^l a_{2-2s}^{(p)}(m,-j)\sum_{h|j}\chi_D(h)h^{s-1}p_{s+1/2}^{(N(n))}\left (-|d|,\tfrac{j^2|D|}{h^2}\right ).
\end{align*}
We will rewrite this by grouping each of the coefficients of the Poincar\'e series together according to level, to obtain
\begin{align*}
-\Tr^*_{d,D}&(f_{2-2s,m}^{(p)})=\sum_{n|m}\chi_1^{(p)}(n)\chi_D(n)n^{s-1}p_{s+1/2}^{(4p)}\left (-|d|,\tfrac{m^2|D|}{n^2}\right )\\
&+\sum_{j=1}^l a_{2-2s}^{(p)}(m,-j)\sum_{h|j}\chi_1^{(p)}(n)\chi_D(h)h^{s-1}p_{s+1/2}^{(4p)}\left (-|d|,\tfrac{j^2|D|}{h^2}\right )\\
&+\chi_D(p)p^{s-1}\sum_{n|(m/p)}\chi_D(n)n^{s-1}\left ( p_{s+1/2}^{(4)}\left (-|d|,\tfrac{(m/p)^2|D|}{n^2}\right ) \right )\\
&+\chi_D(p)p^{s-1}\sum_{j=1}^l a_{2-2s}^{(p)}(m,-j)\sum_{h|(j/p)}\chi_D(h)h^{s-1}\left ( p_{s+1/2}^{(4)}\left (-|d|,\tfrac{(j/p)^2|D|}{n^2}\right )\right ),\\
\end{align*}
where the terms involving $(m/p)$ and $(j/p)$ are taken to be $0$ if $p\nmid m$ or $p\nmid j$ respectively.\\

We will denote $C(z)=F_{s+1/2,-|d|}^{(4p)}(z)-g_{s+1/2,|d|}^{(4p)}(z)$.  We know from section \ref{sec2} that $C(z)$ is a modular form which vanishes at the cusps $(\infty)$ and $(0)$ and is in the plus space, thus is a cusp form in $S_{s+1/2}(4p)$.  Denote $C'(z) = \left (F_{s+1/2,-|d|}^{(4)}(z)-f_{s+1/2,|d|}^{(4)}(z)\right )$, which is a cusp form in $S_{s+1/2}(4)$.  We will denote
$$C(z)=\sum c(n)q^n \,\,\,\,\text{ and }\,\,\,\, C'(z)=\sum c'(n)q^n.$$
Thus equating coefficients gives
$$p_{s+1/2}^{(4p)}\left (-|d|,\tfrac{j^2|D|}{h^2}\right )=b_{s+1/2}^{(4p)}\left (-|d|,\tfrac{j^2|D|}{h^2}\right ) +c\left (\tfrac{j^2|D|}{h^2}\right)$$
and
$$p_{s+1/2}^{(4)}\left (-|d|,\tfrac{j^2|D|}{h^2}\right )=b_{s+1/2}^{(4)}\left (-|d|,\tfrac{j^2|D|}{h^2}\right )+c'\left (\tfrac{j^2|D|}{h^2}\right ).$$

The contributions to $-\Tr^*_{d,D}(f_{2-2s,m}^{(p)})$ from $C(z)$ and $C'(z)$ are
\begin{equation}\label{eqC}
\begin{aligned}
\sum_{n|m}\chi_1^{(p)}(n)\chi_D(n)n^{s-1}c\left (\tfrac{m^2|D|}{n^2}\right)+\sum_{j=1}^l a_{2-2s}^{(p)}(m,-j)\sum_{h|j}\chi_1^{(p)}(n)\chi_D(h)h^{s-1}c\left (\tfrac{j^2|D|}{h^2}\right)
\end{aligned}
\end{equation}
and
\begin{equation}\label{eqC'}
\begin{aligned}
\sum_{n|(m/p)}\chi_D(n)n^{s-1}&c'\left (\tfrac{(m/p)^2|D|}{n^2}\right)\\
&+\sum_{j=1}^l a_{2-2s}^{(p)}(m,-j/p)\sum_{h|(j/p)}\chi_D(h)h^{s-1}c'\left (\tfrac{(j/p)^2|D|}{h^2}\right).
\end{aligned}
\end{equation}

We note that $\mathfrak{S}_D^{(4p)}(C(z))$, the $D$th Shimura lift of level $4p$ of $C(z),$ and $\mathfrak{S}_D^{(4)}(C'(z)),$  the $D$th Shimura lifts of level $4$ of $C'(z)$, are cusp forms of weight $2s$ and level $p$ and $1$ respectively with the $j$th and $(j/p)$th coefficients equal to
$$\sum_{h|j}\chi_1^{(p)}(n)\chi_D(h)h^{s-1}c\left (\tfrac{j^2|D|}{h^2}\right),$$
$$\sum_{h|(j/p)}\chi_D(h)h^{s-1}c'\left (\tfrac{(j/p)^2|D|}{h^2}\right),$$
respectively.  Thus (\ref{eqC}) and (\ref{eqC'}) may be interpreted as the constant term of $\mathfrak{S}_D^{(4p)}(C(z))\cdot f_{2-2s,m}^{(p)}(z)$ and the constant term of $\mathfrak{S}_D^{(4)}(C'(z))|V_p\cdot f_{2-2s,m}^{(p)}(z)$ respectively, both of which are in $M_2^!(\Gamma_0(p))$. Since $\mathfrak{S}_D^{(4p)}(C(z))\cdot f_{2-2s,m}^{(p)}(z)$ and $\mathfrak{S}_D^{(4)}(C'(z))|V_p\cdot f_{2-2s,m}^{(p)}(z)$ vanish at the cusp $(0)$ and since $M_2(p)$ is one dimensional, we use an argument similar to that in section \ref{sec4} to conclude that $\mathfrak{S}_D^{(4p)}(C(z))\cdot f_{2-2s,m}$ and $\mathfrak{S}_D^{(4p)}(C'(z))|V_p\cdot f_{2-2s,m}$ both have zero constant terms, implying that (\ref{eqC}) and (\ref{eqC'}) are both zero.  Therefore,
\begin{align*}
\Tr^*_{d,D}(f_{2-2s,m}^{(p)})=&\sum_{n|m}\chi_D(n)n^{s-1}b_{s+1/2}^{(N(n))}\left (-|d|,\tfrac{m^2|D|}{n^2}\right )\\
&+\sum_{j=1}^l a_{2-2s}^{(p)}(m,-j)\sum_{h|j}\chi_D(h)h^{s-1}b_{s+1/2}^{(N(n))}\left (-|d|,\tfrac{j^2|D|}{n^2}\right ).\\
\end{align*}

Now by duality, $\Tr^*_{d,D}(f_{2-2s,m}^{(p)})$ is the coefficient of $q^{|d|}$ in the Fourier expansion of
$$\sum_{n|m}\chi_D(n)n^{s-1}f_{3/2-s,\frac{m^2|D|}{n^2}}^{(N(n))}(z)-\sum_{j=1}^l b_{2s}^{(p)}(-j,m)\sum_{h|j} \chi_D(h)h^{s-1}f_{3/2-s,\frac{j^2|D|}{h^2}}^{(N(n))}(z)$$

Recall that we denote the modular form $f \in M^\sharp_{2-2s}(p)$ which we are lifting as $$f=\sum a(n)q^n. $$  Since we can express any such $f$ in terms of the basis elements $\{f_{2-2s,m_i}^{(p)}\}$, we write
$$f = \sum_{m\geq l}a(-m)f_{2-2s,m}^{(p)}.$$
Then, by the linearity of the trace,
$$\Tr^*_{d,D}(f)=\sum_{m\geq l}a(-m)\Tr^*_{d,D}(f_{2-2s,m}^{(p)}).$$
Following the above argument, $\Tr^*_{d,D}(f)$ is the coefficient of $q^{|d|}$ in
$$\sum_{m\geq l}a(-m)\sum_{n|m}\chi_D(n)n^{s-1}f_{3/2-s,\frac{m^2|D|}{n^2}}^{(N(n))}(z)-\sum_{j=1}^l b_{2s}^{(p)}(-j,m)\sum_{h|j} \chi_D(h)h^{s-1}f_{3/2-s,\frac{j^2|D|}{h^2}}^{(N(n))}(z).$$

We then change the order of summation and group together all the terms with equal $j$ to rewrite this as
\begin{equation}\label{reduces}
\begin{aligned}
\sum_{m\geq l}a(-m)\sum_{n|m}&\chi_D(n)n^{s-1}f_{3/2-s,\frac{m^2|D|}{n^2}}^{(N(n))}(z)\\
&-\sum_{j=1}^l\sum_{h|j} \chi_D(h)h^{s-1}f_{3/2-s,\frac{j^2|D|}{h^2}}^{(N(n))}(z) \sum_{m\geq l}a(-m) b_{2s}^{(p)}(-j,m).
\end{aligned}
\end{equation}
The fact that $f(z)\cdot g_{2s,-j}^{(p)}\in M_2^!(p)$ and vanishes at the cusp $(0)$ means that its constant term must vanish.  Thus we get
$$-\sum_{m\geq l}a(-m)b_{2s}^{(p)}(-j,m)=a(-j),$$
from which equation (\ref{reduces}) reduces to
$$\sum_{m>0}a(-m)\sum_{n|m}\chi_D(n)n^{s-1}f_{3/2-s,\frac{m^2|D|}{n^2}}^{(N(n))}(z),$$
which is exactly $F$.\\

Now we will look at the principal parts.  The definition of the basis elements $f_{3/2-s,m}^{(N)}(z)$ show that
$$f_{3/2-s,\frac{m^2|D|}{n^2}}^{(N)}(z)=0$$
as long as $\frac{m^2|D|}{n^2}<C_0$ for some constant $C_0$ which only depends on the weight $3/2-s$ and the level $N$.  We combine this fact with the Fourier expansion
$$ f_{3/2-s,\frac{m^2|D|}{n^2}}^{(N)}(z)= q^{-\frac{m^2|D|}{n^2}}+\sum_h a_{3/2-s}^{(N)}\left (\tfrac{m^2|D|}{n^2},h\right )q^h$$
to express the negative powers of $q$ appearing in the Fourier expansion of $F$.  This is
\begin{equation}\label{S}
\begin{aligned}
\sum_{m>0}a(-m)&\sum_{n|m}\chi_D(n)n^{s-1}q^{\frac{-m^2|D|}{n^2}}\\
&-\sum_{m>0}a(-m)\sum_{n|m,\frac{m^2|D|}{n^2}<C_0}\chi_D(n)n^{s-1}q^{\frac{-m^2|D|}{n^2}}\\
&+\sum_{m,h>0}a(-m)\sum_{n|m}\chi_D(n)n^{s-1}a_{3/2-s}^{(N(n))}\left (\tfrac{m^2|D|}{n^2},-h\right)q^{-h}.
\end{aligned}
\end{equation}

The first line of the above expression (\ref{S}) is exactly what we want for the principal part of $\mathfrak{Z}_Df$.  Thus we must prove that the rest of (\ref{S}) vanishes.  We will call this $S$.  Coefficient duality shows that
$$S=-\sum_{m>0}a(-m)\left[\sum_{n|m,\frac{m^2|D|}{n^2}<C}\chi_D(n)n^{s-1}q^{\frac{-m^2|D|}{n^2}}+\sum_{h>0}\sum_{n|m} \chi_D(n)n^{s-1}b_{s+1/2}^{(N(n))}\left (-h,\tfrac{m^2|D|}{n^2}\right )q^{-h}\right].$$
Again, we rewrite this sum by grouping together coefficients which come from basis elements of the same level to get
\begin{equation}
\begin{aligned}
S=-\sum_{m>0}&a(-m)\Bigg[\sum_{n|m,\frac{m^2|D|}{n^2}<C}\chi_1^{(p)}(n)\chi_D(n)n^{s-1}q^{\frac{-m^2|D|}{n^2}}\\&+\sum_{h>0}\sum_{n|m}\chi_1^{(p)}(n) \chi_D(n)n^{s-1}b_{s+1/2}^{(4p)}\left (-h,\tfrac{m^2|D|}{n^2}\right )q^{-h}\\
&+\chi_D(p)p^{s-1}\sum_{h>0}\sum_{n|(m/p)} \chi_D(n)n^{s-1}\left (b_{s+1/2}^{(4)}\left (-h,\tfrac{(m/p)^2|D|}{n^2}\right )\right )q^{-h}\Bigg].
\end{aligned}
\end{equation}

For any $h>0$, the coefficient of $q^m$ in the $D$th Shimura lift $\mathfrak{S}_D^{(4p)}\left (g_{s+1/2,-h}^{(N)}\right )$ of level $N$ with $4|N$ of the cusp form $g_{s+1/2,-h}^{(N)}$ is given by
$$\sum_{n|m}\chi_{1}^{(N/4)}(n)\chi_D(n)n^{s-1}\left[b_{s+1/2}^{(N)}\left (-h,\tfrac{m^2|D|}{n^2}\right)+\{\text{1 or 0}\}\right].$$

The last term here arises from the initial $q^h$ in the Fourier expansion of $g_{s+1/2,-h}^{(N)}$, since $b_{s+1/2}^{(N)}(-h,h)$ is zero by definition; it equals 1 if $\tfrac{m^2|D|}{n^2}=h$ and 0 otherwise.  Thus we interpret the coefficient of $q^{-h}$ in $S$ for each $h>0$ as the sum of constant terms of $\mathfrak{S}_D^{(4p)}(g_{s+1/2,-h}^{(4p)})\cdot f$ and $\mathfrak{S}_D^{(4)}(g_{s+1/2,-h}^{(4)})|V_p\cdot f$.  Both of these forms are in $\Gamma_2^!(p)$ and vanish at the cusp $(0)$ and thus have vanishing constant terms.  Therefore $S=0$ and we have equality with the principal part of $F$.  The concludes the proof.\end{proof}

We now prove the analogue of (\ref{maintheorem}) for $(-1)^sD<0$.

\begin{theorem}\label{maintheorem2}
Let $f\in M^\sharp_{2-2s}(p)$ and $s\geq 2$ be an integer.  Suppose that $D$ is a fundamental discriminant with $(-1)^sD>0$.  Then the $D$th Zagier lift of $f$ is given by
\begin{equation}\label{Fequation2}
\mathfrak{Z}_Df=\sum_{m>0} a(-m)m^{2s-1}\sum_{n|m} \chi_D(n) n^{-s}f_{s+1/2,\frac{m^2|D|}{n^2}}^{(N(n))}(z)+C(z),
\end{equation}
where $C(z)\in M_{s+1/2}(4p)$ is the unique modular form whose Fourier coefficients match those of $\mathfrak{Z}_D(f)$ for the first $l$ positive values of $n$ with $(-1)^sD\equiv 0,1\pmod 4$.
\end{theorem}
\begin{proof}
Using 1.1 and 1.2 from Miller and Pixton \cite{Miller} as before, we have
\begin{equation}
\begin{aligned}
-\Tr^*_{d,D}(f_{2-2s,m}^{(p)})=&m^{2s-1}\sum_{n|m}\chi_D(n)n^{-s}p_{s+1/2}^{(N(n))}\left (-|d|,\tfrac{m^2|D|}{n^2}\right )\\
&+\sum_{j=1}^l a(f_{2-2s,m},-j)j^{2s-1}\sum_{h|j}\chi_D(h)h^{-s}p_{s+1/2}^{(N(n))}\left (-|d|,\tfrac{j^2|D|}{n^2}\right ).
\end{aligned}
\end{equation}

Then, writing an arbitrary form $f(z)\in M_{2-2s}^!(p)$ in terms of basis elements
$$f(z)=\sum_{m\geq l}a(-m)f_{2-2s,m}^{(4p)}(z)$$
shows that the trace of $f$ is given by

\begin{equation}\label{traceodd}
\begin{aligned}
\sum_{m\geq l}a(-m)\Bigg [&m^{2s-1}\sum_{n|m}\chi_D(n)n^{-s}p_{s+1/2}^{(N(n))}\left (-|d|,\tfrac{m^2|D|}{n^2}\right )\\
&+\sum_{j=1}^l a_{2-2s}^{(p)}(m,-j)j^{2s-1}\sum_{h|j}\chi_D(h)h^{-s}p_{s+1/2}^{(N(n))}\left (-|d|,\tfrac{j^2|D|}{n^2}\right ) \Bigg ].
\end{aligned}
\end{equation}

As before, we use the fact that
$$-\sum_{m\geq l}a(-m)a_{2s}^{(p)}(-j,m)=a(-j)$$
to rewrite the above equation \ref{traceodd} as
$$\sum_{m>0}a(-m)m^{2s-1}\sum_{n|m}\chi_D(n)n^{-s}p_{s+1/2}^{(N(n))}\left (-|d|,\tfrac{m^2|D|}{n^2}\right ).$$
But this is just the coefficient of $q^{|D|}$ in the modular form given by
$$G(z)=\sum_{m>0}a(-m)m^{2s-1}\sum_{n|m}\chi_D(n)n^{-s}F_{s+1/2,-|d|}^{(N(n))}(z).$$
Then, let $F_{s+1/2,-|d|}^{(N(n))}(z)-f_{s+1/2,|d|}^{(N(n))}(z)=C(z)$, a form in $M_{s+1/2}(4p)$.  So we can write
$$G(z)=\sum_{m>0}a(-m)m^{2s-1}\sum_{n|m}\chi_D(n)n^{-s}f_{s+1/2,|d|}^{(N(n))}(z)+C(z).$$
Arguing as in the case where $(-1)^sD>0$, the principal part of $F$ matches the principal part of $\mathfrak{Z}_D(f)$.  The negative terms match as well, concluding the proof.\\
\end{proof}

\begin{remark}
We note that although the constant term $C$ is not explicitly given in (\ref{zaglift}), it can be explicitly calculated by expressing the image of the lift using basis elements $f_{k,m}^{(N)}$.  We note that Duke and the second author give an explicit formula for $C$ when $N=4$ using a Shimura lift argument similar to that used in the proof of (\ref{maintheorem}).  Their technique does not extend to levels $4p$, because the modified Shimura lift does not preserve zeros at the cusps when the lift is extended to non-cuspidal forms.
\end{remark}

The statement on integrality in theorem \ref{theorem:ours} follows directly from theorem \ref{maintheorem} in the case where $(-1)^sD>0$.  In the case where $(-1)^sD<0$, the statement on integrality reduces to the case where $(-1)^sD>0$ using the following identity, which holds if $(-1)^sD<0$ and $D'$ is a fundamental discriminant with $(-1)^sD'>0$:
$$\Tr^*_{m^2D',D}(f)=-m^{2s-1}\sum_{a|m}\mu(a) \chi_{D'}(a) \sum_{b|(m/a)} \chi_D(b)(ab)^{-s} \Tr^*_{(\tfrac{m}{ab})^2D,D'}(f).$$
This identity is a consequence of the following lemma.
\begin{lemma}
For $D$ and $D'$ fundamental discriminants with $DD'<0$ and $m\in \Z^+$ we have
$$\mathrm{Tr}_{m^2D',D}(f)=\sum_{a|m}\mu(a) \chi_{D'}(a) \sum_{b|(m/a)} \chi_D(b) \mathrm{Tr}_{(\tfrac{m}{ab})^2D,D'}(f).$$
\end{lemma}

\begin{proof}
The proof is identical to the proof given for lemma 2 in \cite{Jenkins}.\\

\end{proof}

\bibliographystyle{amsplain}

\end{document}